\documentclass{amsart}
\usepackage{amssymb, verbatim, amsthm, amsfonts, amsmath}
\usepackage[all]{xy}

\newtheorem{lemma}{Lemma}
\newtheorem{thm}{Theorem}
\newtheorem{cor}{Corollary}

\newtheorem{conj}{Conjecture}

\newcommand{\fq}{\mathbb{F}_q}
\newcommand{\fqn}{\mathbb{F}_{q^n}}
\newcommand{\F}{\mathbb{F}}
\newcommand{\rmv}[1]{}

\begin{document}

\title[Complete permutation monomials]{Further results on complete permutation monomials over finite fields}

\author[Feng]{Xiutao Feng}
\address{Key Laboratory of Mathematics Mechanization, Academy of Mathematics and Systems Science, Chinese Academy of Sciences, Beijing 100190, China}
\email{fengxt@amss.ac.cn.}

\author[Lin]{Dongdai Lin}
\address{State Key Laboratory of Information Security, Institute of Information Engineering, Chinese Academy of Sciences, Beijing 100093, China.}
\email{ddlin@iie.ac.cn}

\author[Wang]{Liping Wang}
\address{State Key Laboratory of Information Security, Institute of Information Engineering, Chinese Academy of Sciences, Beijing 100093, China}
\email{wangliping@iie.ac.cn}

\author[Wang]{Qiang Wang}
\thanks{The research of authors  is partially supported by National Natural Science Foundation of China (No. 61379139), National Natural Science Foundation of China (No. 61572491), Foundation Science Center project of the National Natural Science Foundation of China (No. 11688101), and  NSERC of Canada.}
\address{School of Mathematics and Statistics,
Carleton University, Ottawa, Ontario, $K1S$ $5B6$, CANADA}

\email{wang@math.carleton.ca}

\keywords{\noindent finite fields, permutation polynomials, complete permutation polynomials, monomials} \subjclass[2010]{11T06, 05A05,11T55, 94B25}

\begin{abstract}
In this paper, we construct  some new classes of complete permutation monomials with exponent $d=\frac{q^n-1}{q-1}$ using AGW criterion (a special case).  This  proves two recent conjectures in \cite{Wuetal2} and extends some of these recent results to more general $n$'s. 
\end{abstract}

\maketitle

 Let $q = p^k$ be the power of a prime number $p$, $\F_q$ be a finite field with $q$ elements, and  $\F_q[x]$ be the ring of polynomials over $\F_q$. 
 We call $f(x) \in \F_q[x]$ a {\em permutation polynomial} (PP) of $\F_q$ 
 if  $f$ induces a permutation of $\F_q$. 
An {\em exceptional polynomial} of $\F_q$ is a polynomial $f\in \F_q[x]$ which is a permutation polynomial over $\F_{q^m}$ for infinitely many $m$. It is well known that a permutation polynomial over $\F_q$ of degree at most $q^{1/4}$  is exceptional over $\F_q$. For more background material on permutation
polynomials we refer to Chap. 7 of \cite{LN:97}. For a detailed survey of open questions and recent results see \cite{Hou:15}, \cite{LM:88},
\cite{LM:93},  \cite{Mullen:93}, and \cite{MullenWang:12}.

A complete permutation polynomial (CPP) is a polynomial $f(x)$ satisfies that both $f(x)$ and $f(x)+x$ induce bijections of $\fq$.  
CPPs have recently become a strong source of interest due to their connection to combinatorial objects such as orthogonal Latin squares 
and due to their applications in cryptography; in particular, in the construction of bent functions \cite{nyberg,  ribic, stanica}.
See also \cite{evans, tu, tuxanidy, wu and lin} and the references therein for some recent work in the area.

The most studied class of CPPs are monomials.  If there exists a complete permutation monomial of degree $d$ over $\fq$, then $d$ is called a CPP exponent over $\fq$.  Earlier work has been done recently in \cite{Bartolietal1, Bartolietal2, BZ1, BZ2, Maetal, Wuetal1}.  
Thesel recent papers concentrated on classifying monomials $f(x)=a^{-1}x^d$ with special $d$ that are  CPPs.  In particular,  $d=  \frac{q^n-1}{q-1}+1$ implies that $(d, q^n-1) =1$ and thus   $x^d$ is always a PP of $\F_{q^n}$.   Essentially, this problem  is equivalent to classifying  permutation binomial of the form $x^d + ax$.   The main focus of the paper is therefore to determine when certain polynomials of the form $f(x):=x^{\frac{q^n-1}{q-1}+1}+bx$  are  permutation binomials over $\fqn$. We remark earlier references on permutation binomials can be found in \cite{Hou:15a, MullenWang:12} and reference therein.

One of useful criterions in studying PPs of finite fields is the following AGW criterion. It first appeared in  \cite{AGW:11} and further developed  in \cite{YuanDing:11, YuanDing:14}, among others.

\begin{lemma}[AGW criterion]\label{mainlemma}
Let $A$, $S$ and $\bar{S}$ be finite sets with $\# S = \# \bar{S}$, and let $f:A\rightarrow A$, $\bar{f}:S\rightarrow \bar{S}$, $\lambda:A\rightarrow S$, and
$\bar{\lambda}:A\rightarrow \bar{S}$  be maps such that $\bar{\lambda} \circ f=\bar{f} \circ\lambda$.

\[
\xymatrix{
A \ar[r]^{P}\ar[d]^{\lambda} & A \ar[d]^{\bar{\lambda}}\\
S \ar[r]^{\bar{P}} & \bar{S} }
\]

If both $\lambda$ and $\bar{\lambda}$ are surjective, then the following statements are equivalent:
\begin{enumerate}
\item[(i)] $f$ is a bijection (a permutation of $A$); and
\item[(ii)] $\bar{f}$ is a bijection from $S$ to $\bar{S}$  and $f$ is injective on $\lambda^{-1}(s)$ for each $s\in S$.
\end{enumerate}
\end{lemma}

We remark that  $P(x)$ in Lemma~\ref{mainlemma} can be viewed piecewisely.  Namely,  let $S=\{s_0, s_1, \ldots, s_{\ell-1}\}$, then we have

$$
P(x) = 
\left\{
\begin{array}{ll} 
P_0(x), &  if~ x \in C_0 = \lambda^{-1}(s_0); \\
\vdots &  \vdots \\ 
P_i (x),   &   if ~x \in C_i = \lambda^{-1}(s_i) ; \\
\vdots & \vdots \\
P_{\ell -1}(x),  &  if~ x \in C_{\ell-1} = \lambda^{-1}(s_{\ell-1}),
\end{array}
\right.
$$ 
is a bijection if and only if   each $P_i$ is injective on $C_i$ for $0\leq i \leq \ell-1$ and $\bar{P}$ is a bijection.

In particular, if we take $A  =  \fq^* = < \gamma>$,  $\ell s = q-1$,  $\zeta = \gamma^s$, and  $\lambda=\bar{\lambda} = x^s$, then we obtain the following 
diagram.

\[
\xymatrix{
\fq^* \ar[r]^{P}\ar[d]^{x^s} & \fq^* \ar[d]^{x^s}\\
S=\{1, \zeta, \ldots, \zeta^{\ell-1} \} \ar[r]^{\bar{P}} & S = \{1, \zeta, \ldots, \zeta^{\ell-1} \}}
\]

This special case of AGW criterion was obtained earlier in \cite{NW:05, Wang}.  Namely,  the cyclotomic mapping polynomial 
$$
P(x) = 
\left\{
\begin{array}{ll} 
A_0 x^r, &  if~ x \in C_0 =<\gamma^\ell> \leq \fq^* = <\gamma>; \\
\vdots &  \vdots \\ 
A_i x^r,   &   if ~x \in C_i = \gamma^i C_0; \\
\vdots & \vdots \\
A_{\ell -1} x^r,  &  if~ x \in C_{\ell-1} = \gamma^{\ell-1} C_0,
\end{array}
\right.
$$ 
 is a PP of $\fq$ if and only if $(r, s) =1$ and $\bar{P}$ permutes $S= \{1, \zeta, \ldots, \zeta^{\ell-1} \}$. 

Rewriting it in terms of polynomials,  this criterion essentially is the following useful criterion appeared in \cite{LeePark:97, Wang, Zieve:09} in different forms.

\begin{cor}[Park-Lee 2001, Wang 2007, Zieve 2009]\label{lemma}
Let  $q-1 = \ell s$ for some positive
integers $\ell$ and $s$.  Then $P(x) = x^r f(x^s)$ is a PP of $\mathbb{F}_q$  if and 
only if   $(r, s) =1$ and $x^r f(x)^s$ permutes the set $\mu_{\ell}$ of all distinct $\ell$-th roots of unity.
\end{cor}

Using Corollary~\ref{lemma}, Akbary and Wang \cite{AW:07} first studied the permutation polynomials of the form $x^rf(x^s)$ for  arbitrary $\ell$ and  those polynomials $f$'s  satisfying that $f(\zeta)^s = \zeta^j$ for some $j$,  where $\zeta \in \mu_{\ell}$.  In particular, 

\begin{cor}\label{power}
Let $q-1=\ell s$.  Assume that $f(\zeta^i)^s =1$ for any $ i=0,\ldots, \ell -1$  and $\zeta \in \mu_{\ell}$ is a primitive $\ell$-th root of unity. Then  $P(x)=x^rf(x^s)$ is a PP of $\fq$  if and only if  $(r , q-1) = 1$.
\end{cor}

More classes of those permutation polynomials with small $\ell$ and special $f$ such that $f(x) = x^e +a^s$ or $f(x) =x^k + x^{k-1} + \cdots + x + 1$ were studied earlier in \cite{AAW:08, AW:05, AW:06, AW:07} and slightly extended in \cite{Zieve-2, Zieve:09}.  PPs of large indices was studied in \cite{Wang:13}.   For the intermediate indices,  Zieve first considered special $\ell = q-1$ or $\ell = q+1$ over finite field $\F_{q^2}$ \cite{Zieve:13}; see also in \cite{Hou:15a}.  We note that several recent papers deal with PPs with this type of indices, see for example \cite{LiQuChen:17, ZhaHuFan:17}.

In the study of CPP monomials, Wu et al  \cite{Wuetal1, Wuetal2} studied the CPP monomials $a^{-1} x^d$ over $\F_{p^{nk}}$  such that 
$d = \frac{p^{nk}-1}{p^k-1} + 1$.   For any $a \in \F_{p^{nk}}$, let $a_i = a^{p^{ik}}$, where $0\leq i \leq n-1$. Define 
\[
h_a(x) = x \prod_{i=0}^{n-1} (x+a_i). 
\]

Then  Corollary~\ref{lemma} directly gives the following

\begin{cor}\label{lemma2}
Let $d=\frac{p^{nk}-1}{p^k-1} + 1$.   Then $x^d + a x \in \F_{p^{nk}}[x]$ is a PP of $\F_{p^{n k}}$ if and only if $h_a(x) \in \F_{p^k}[x]$ is a PP of $\F_{p^k}$. 
\end{cor}

In this case  $h_a(x) = x (x+a)^{d-1}$ is a  polynomial of lower degree over $\mu_{p^k-1}$ or $\F_{p^k}$.  When $n$ is small, essentially we need to study permutation polynomial of low degree over a subfield $\F_{p^k}$.   In \cite{BZ1, BZ2, Wuetal1, Wuetal2},  PPs of the form $f_a(x) = x^d + ax$ over $\fqn$ are thoroughly investigated for $n = 2,  3, 4$. For $n = 6$, sufficient conditions for $f_a(x)$  to be a PP of $\F_{q^6}$ are provided in \cite{Wuetal1, Wuetal2}  in the special cases of characteristic  $p\in \{2, 3, 5\}$, whereas in \cite{Bartolietal1} all $a$'s for which $ax^{\frac{q^6-1}{q-1}} +1$ is a CPP over $\F_{q^6}$ are explicitly listed.   The case $p = n+ 1$ is dealt with in \cite{Wuetal2, Maetal} as well.

 The following two conjectures are made in \cite{Wuetal2}.  

\begin{conj}[Conjecture 4.18 in \cite{Wuetal2}]\label{conj1}
Let $n+1$ be a prime such that $n+1 \neq p$. Let $(n, k)=1$ and $(n+1, p^2-1)=1$, and 
$d = \frac{p^{nk}-1}{p^k-1} + 1$.  Then there exists $a \in \F_{p^{nk}}^*$  such that $h_a(x)$ are Dickson polynoials of degree $n+1$ over $\F_{p^k}$. 
\end{conj}

\begin{conj}[Conjecture 4.20 in \cite{Wuetal2}]\label{conj2}
Let $p$ be an odd prme. Let $n+1 = p$ and $d = \frac{p^{nk}-1}{p^k-1} + 1$. Then $a^{-1} x^d$ is a CPP over $\F_{p^{nk}}$, where $a\in \F_{p^{nk}}^*$ such that $a^{p^k-1} = -1$. 
\end{conj}

Using the classification results of exceptional polynomials,  Bartoli et al \cite{Bartolietal2} classified complete permutation monomials of degree $\frac{q^n-1}{q-1} + 1$ over the finite field with $q^n$ elements in odd characteristic,
for $n +1$ a prime and $(n + 1)^4 < q$. As a corollary, Conjecture~\ref{conj1}  was proven in odd characteristic. However,   Conjecture~\ref{conj1} is still open in general.  We note Conjecture~\ref{conj2}   is  proven recently in \cite{Maetal} for $n+1$ prime.  However, when $n+1$ is large or not prime, the classification of CPP monomials  is still open. For example, when $n+1$ is a power of primes such as $8$ or $9$,  only a few new examples of CPPs are provided in \cite{Bartolietal2}.   In this paper, we construct several new classes of CPPs using AGW criterion,  which confirm both conjectures.  In Section 1, we use a factorization result of Dickson polynomial to construct explicitly a new class of CPPs and prove Conjecture~\ref{conj1}. 
In comparision to the proof in \cite{Bartolietal2},  our result is more explicit although our result is not a classification result.  However our proof is elementary and it does not assume that $(n+1)^4 < p^k$. 
 In Section 2, we derived a few new classes of CPPs for general $n$ such that $n\mid p-1$ or $n\mid p^k-1$ respectively. This covers the case of Conjecture~\ref{conj2} and also gives  new classes of CPPs with large $n$.  We also demonstrated a usage of AGW criterion in proving a well known class of degree $p$ exceptional polynomials. 

\section{CPPs induced from Dickson polynomials}

Let us consider Dickson polynomial of the first kind $D_{n}(x, b)$ of degree $n$ over $\F_q$. It is well known that 
 $D_{n+1}(x, b) = D_{n+1} (y + b/y, b) = y^{n+1} + (b/y)^{n+1}$ when we let $x = y+b/y$ for some $y\in \F_{q^2}$.  Therefore $w = u+ \frac{b }{u}$ is a root of $D_{n+1}(x, b)$ if and only if $u^{n+1} +  \frac{b^{n+1}}{u^{n+1}} =0$. Equivalently, $u$ is a solution to $u^{2(n+1)} = - b^{n+1}$.  If $c^2 = b$ and
$\zeta$ is a primitive $4(n+1)$-th root of unity, then $u$ is of the form $c \zeta^{2i+1}$ for $0\leq i \leq n-1$.  Therefore all the roots of $D_{n+1}(x, b)$ are $ c (\zeta^{2i+1} + \zeta^{-(2i+1)})$ for $0\leq i \leq n$. Moreover,  the explicit factorization for Dickson polynomial of the first kind was studied earlier in \cite{Chou} and \cite{BhargavaZieve}. See also Theorem 9.6.12 in \cite{HFF} as follows:

\begin{thm}\label{factor}  If $q$ is odd and $a\in \fq^*$ then $D_n(x,a)$ is the product of irreducible polynomials over $\F_q$ which occur in cliques corresponding to the divisors $d$ of $n$ for which $n/d$ is odd. 
Let $m_d$ is the least positive integer satisfying $q^{m_d}\equiv \pm 1\pmod{4d}$.
To each such $d$ there corresponds $\phi(4d)/(2N_d)$ irreducible factors of degree $N_d$, each of which has the form 
$$\prod_{i=0}^{N_d-1}(x-\sqrt{a^{q^i}}(\zeta^{q^i}+\zeta^{-q^i}))$$
where $\zeta$ is a primitive $4d$-th root of unity and
$$N_d=\left\{ \begin{array}{ll}
m_d/2 & \mbox{if $\sqrt{a}\notin \F_q , m_d \equiv 2 ~(\bmod{~ 4})$ and $q^{m_d/2} \equiv 2d\pm 1~(\bmod{~ 4d})$},\\
2m_d & \mbox{if $\sqrt{a}\notin \F_q$ and $m_d$ is odd},\\
m_d & \mbox{otherwise}. 
             \end{array}
             \right..$$ 
\end{thm}

We use this factorization result to construct a class of CPP as follows:

\begin{thm}\label{thm1}
Let $p$ be an odd prime,  $k$ be an odd positive integer, and  $n+1$ be an odd prime number such that $(n, k)=1$ and $(n+1, p^2-1) =1$.   Let
$d = \frac{p^{nk}-1}{p^k-1} + 1$,  $b\in \F_{p^k}^*$, and  $c\in \F_{p^{2k}}$ such that $b=c^2$. Let $a=c(\zeta + \zeta^{-1})$ and $\zeta$ is a primitive  $4(n+1)$-th  roots of unity in  $\F_{p^{nk}}$.  Then 

(i) If  $p\equiv 1\pmod 4$,  or $p \equiv 3 \pmod{4}$ and $n/2$  is even, then $a^{-1} x^d$ is a CPP over $\F_{p^{nk}}$. 

(ii) If $p \equiv 3 \pmod{4}$ and $n/2$  is odd,  then $a^{-1} x^d$ is a CPP over $\F_{p^{nk}}$ for all $c \in \F_{p^{2k}} \setminus  \F_{p^k}$ and $b=c^2 \in \F_{p^k}$. 
\end{thm}

\begin{proof}

Let $q=p^k$ and $m_d$ be the smallest positive integer such that $q^{m_d} \equiv \pm 1 \pmod{4(n+1)}$.  By Fermat's little theorem, we obtain $p^{n} \equiv 1 \pmod{n+1}$ and $p^{n/2} \equiv - 1\pmod{n+1}$.  Because $(n, k) =1$, we must have $k$ odd. If $p\equiv 1\pmod 4$, then $q=p^k \equiv 1 \pmod{4}$. Therefore the smallest positive integer $m_d$ such that $q^{m_d} \equiv \pm 1 \pmod{4(n+1)}$ is $m_d = n$.  If $p \equiv 3 \pmod{4}$, then $q \equiv -1 \pmod{4}$ because $k$ is odd.  Furthermore, if $n/2$ is odd, then   the smallest positive integer $m_d$ such that $q^{m_d} \equiv \pm 1 \pmod{4(n+1)}$ is $m_d = n/2$. If   $n/2$ is even, then   the smallest positive integer $m_d$ such that $q^{m_d} \equiv \pm 1 \pmod{4(n+1)}$ is $m_d = n$. In any case,  $a^{p^{nk}}=c^{p^{nk}} (\zeta^{p^{nk}} + \zeta^{-p^{nk}})= c(\zeta + \zeta^{-1}) = a $ and thus $a \in \F_{p^{nk}}$. 

If   $p\equiv 1\pmod 4$,  or $p \equiv 3 \pmod{4}$ and $n/2$  is even,  then $m_d =n$ and thus by Theorem~\ref{factor} we obtain, for any $b \in \fq^*$, 
$$D_{n+1}(x, b) =x \prod_{i=0}^{n-1}(x-\sqrt{b^{q^i}}(\zeta^{q^i}+\zeta^{-q^i})). $$
If $p \equiv 3 \pmod{4}$ and $n/2$  is odd, then $m_d = n/2$. In this case, let $b = c^2$ such that $c \in \F_{q^2} \setminus \fq$. Then 
$$D_{n+1}(x, b) = x\prod_{i=0}^{n-1}(x-\sqrt{b^{q^i}}(\zeta^{q^i}+\zeta^{-q^i})). $$

In order to prove $x^d + a x$ is a permutation polynomial over $\F_{p^{nk}}$,   by Corollary~\ref{lemma},   $x^d + a x$ is a PP of $\F_{p^{nk}}$ if and only if  $x(x+a)^{d-1}$  permutes $\F_{p^k}^*$. 
We note that 
\begin{eqnarray*}
x(x+a)^{d-1} &=& x(x+a)^{1 + p^k + \cdots + p^{(n-1)k}} \\
&=& x(x+a) (x+a)^{p^k}  \cdots (x+a)^{p^{(n-1)k}} 
\end{eqnarray*}

Therefore $x(x+a)^{d-1} = D_{n+1}(x, b)$ under both assumptions.   Because $(n, k)=1$ and $(n+1, p^2-1)$, it is straightforward to show that $(n+1, p^{2k} -1) =1$. Therefore  $D_{n+1}(x, b)$  permutes $\F_{p^k}^*$. Hence we complete the proof. 
\end{proof}

We remark Theorem~\ref{thm1}  confirms Conjecture~\ref{conj1} by providing explicit choices of $a$.  Our result is more explicit and it does not assume that $(n+1)^4 < p^k$ which is the case in \cite{Bartolietal2}.

 \section{CPP exponent $d=\frac{q^{n}-1}{q-1}$ with more general $n$}

  In this section,  we  first use AGW criterion to give an elementary proof for the exceptionality of the class of degree $p$ exceptional polynomials over $\F_{p^m}$ studied by Fried, Guralnick and Saxl \cite{FGS}.  
Use this class of exceptional polynomials, we derive a new class of complete permutation monomials $a^{-1}x^d$  over $\F_{q^{r}}$ with $d= \frac{q^{p-1}-1}{q-1} + 1$, $q=p^k$,  $a \in \mathbb{F}_{q^n}$ and 
$a^{q-1} \in \mu_n \setminus\{1\} $ where $n \mid p-1$. This  confirms and extends Conjecture~\ref{conj2}.  Similarly, we drive a class of complete permutation monomials over $\F_{q^{n}}$ with exponent $d=\frac{q^{q-1}-1}{q-1} + 1$ and $n\mid q-1$. This provides more examples of CPPs with bigger values of  $n$ while previous studies  mostly deal with small $n$'s or $n=p-1$.    Finally a sufficient and necessary description for CPPs with $d=\frac{p^{nk}-1}{p^k-1}$ is provided. In particular, a simple class of such CPPs is given as a corollary.

\begin{lemma}\label{exceptionalLemma}
Let $p$ be an odd prime, $i\geq 0$, $m, s$ be positive integers such that $ n \mid p-1$. If $c\in \F_{p^m}$ such that $c^{(p^m-1)/n} \neq 1$, then $f(x) = x (x^n - c)^{\frac{p-1}{n}}$ is a permutation polynomial over $\F_{p^m}$. In particular, it is an exceptional polynomial over $\F_{p^m}$ for $m \geq 4$. 
\end{lemma} 
 
 \begin{proof}
 It is clear that $f(x) =0$ if and only if $x =0$. Let $A = \F_{p^m}^*$,  $\lambda(x) = x^n -c$ and $\bar{\lambda} (x) = x^n$, $S = \lambda(A)$, $\bar{S} = \bar{\lambda}(A)$, and $\bar{f} (x) =  (x+c)x^{p-1}$. To show $f$ is a PP of $\F_{p^m}$, by AGW criterion, we need to show that  $f$ is injective on $\lambda^{-1}(s)$ for each $s \in S$ and  $\bar{f}$ is bijective. 
 
 For each $s\in S$, we have $\lambda^{-1}(s)$ is the set of all the distinct roots $\eta_1, \ldots, \eta_n$ of $x^n = s+c$. Hence $f$ is injective on it because $f(\eta_i) = \eta_i s^{\frac{p-1}{n}}$. 
 
 To prove $\bar{f}$ is bijective from $S$ to $\bar{S}$,  is is enough to  show that $\bar{f}$ is surjective.   Let $b \in \bar{S}$ such that $\bar{f}(x)  = (x+c)x^{p-1} = x^p + cx^{p-1} = b$.  Then $x^p + cx^{p-1} = b$ has at least two distinct solutions if and only if 
 $y^p - \frac{c}{b} y -\frac{1}{b}$ has at least two distinct solutions.  Assume $y_1, y_2$ are two distinct solutions in $S$. Then $y_1 -y_2$ is a root of $y^p - \frac{c}{b} y =0$ and thus a root of $y^{p-1} = \frac{c}{b}$, which is a contradiction because $n\mid p-1$, $b \in \bar{S}$, and $c$ is not an $n$-th power.  Therefore there is at most one solution in $S$. Because $|S|=|\bar{S}|$, we conclude $\bar{f}$ is a bijection. Moreover, $f$ is an exceptional polynomial  over $\mathbb{F}_{p^m}$ for $m \geq 4$ because the degree of $f$ is $p$. Hence we complete the proof. 
 \end{proof}

\begin{thm}
Let $p$ be an odd prime, $i\geq 0$, $k, n$ be positive integers such that $n \mid p-1$.   Let
$d = \frac{p^{(p-1)k}-1}{p^k-1} + 1$ and  $a\in \F_{p^{rk}}$ such that $a^{p^k-1} \in \mu_{n}$, where $\mu_{n}$ is the set of all $n$-th roots of unity in $\F_{p^{nk}}$.  If  $a^{p^k-1} \in \mu_{n} \setminus \{1\}$, then  $a^{-1} x^d$ is a CPP over $\F_{p^{nk}}$.
\end{thm} 

\begin{proof}
By Corollary~\ref{lemma} or Corollary~\ref{lemma2}, we need to prove $ h_a(x) = x(x + a)^{d-1}$ is a PP of  $\mu_{p^{k}-1}$ (equivalently, a PP of  $\F_{p^k}$).   If  $a^{p^k-1} \in \mu_{n} \setminus \{1\}$, then $a^{p^k} = a \zeta$ for some $\zeta \neq 1 \in \mu_{n}$.  Then  $a^{p^{ik}} = a \zeta^{i}$ and 
\begin{eqnarray*}
h_a(x)  &=& x(x+a)^{d-1}\\
&=& x(x+a)^{1 + p^k + \cdots + p^{(p-1)k}} \\
&=& x(x+a) (x+a)^{p^k}  \cdots (x+a)^{p^{(p-2)k}} \\
&=& x(x+a) (x+a^{p^k})  \cdots (x+a^{p^{(p-2)k}}) \\
&=& x \left( (x+a)(x+a\zeta) \cdots (x+a\zeta^{n-1}) \right)^{(p-1)/n}
\end{eqnarray*}

If $a^{p^k-1} = -1$, then $h_a(x) = x(x^2-a^2)^{(p-1)/2}$. We note that $a^{p^k-1} =-1$ implies that $a^2 \in \F_{p^k}^*$ but $a\not\in \F_{p^k}$.   By Lemma~\ref{exceptionalLemma},  $h_a(x)$ is a PP of $\F_{p^k}$ and thus $a^{-1} x^d$ is a CPP over $\F_{p^{nk}}$ in this case. 

 If   $a^{p^k-1} \in \mu_{n} \setminus \{1, -1\}$, then $h_a(x) = x (x^n + a^n)^{(p-1)/n}$. 
We note that $a^{p^k-1} \in \mu_{n} \setminus \{1\}$ implies that $a^n \in \F_{p^k}^*$ but $a\not\in \F_{p^k}$.   If  $(p^k-1)/n$ is even, then  $(-a^n)^{(p^k-1)/n } = a^{p^k-1} \neq 1$.  If $(p^k-1)/n$ is odd, then  $(-a^n)^{(p^k-1)/n} = - a^{p^k-1} \neq 1$.  Hence $-a^n$ is not an $n$-th power in $\F_{p^k}$.   By Lemma~\ref{exceptionalLemma},  $h_a(x)$ is a PP of $\F_{p^k}$ and thus we complete the proof. 
\end{proof}

We note that Conjecture 4.20 in   \cite{Wuetal2} was confirmed by Theorem 3.2 \cite{Maetal}. It  is a special case  when $n=p-1$ and $a^{p^k-1} = -1$.  Our result  deals with more general $n$ such that $n \mid p-1$ and  more general $a$'s such that $a^{p^k-1} \in \mu_n \setminus \{1, -1\}$ when $n > 2$.  Next we extend the result for more general $n$'s.

\begin{thm}
Let $p$ be an odd prime, $k, n$ be positive integers such that $n \mid p^k-1$.   Let
$d = \frac{p^{(p^{k}-1) k}-1}{p^k-1} + 1$ and $\ell n = p^{k}-1$. Let  $\mu_{n}$ be the set of  all $n$-th  roots of unity in $\F_{p^{nk}}$.   If  $a^{p^k-1} \in \mu_{n} \setminus \{1\}$, then  $a^{-1} x^d$ is a CPP over $\F_{p^{nk}}$.

\end{thm} 

\begin{proof}
Let $a^{p^k} = a \zeta$ for some $\zeta \neq 1 \in \mu_{n}$. Then  $a^{p^{ik}} = a \zeta^{i}$ and 
\begin{eqnarray*}
h_a(x)  &=& x(x+a)^{d-1}\\
&=& x(x+a)^{1 + p^k + \cdots + p^{(p^k-2)k}} \\
&=& x(x+a) (x+a)^{p^k}  \cdots (x+a)^{p^{(p^k-2)k}} \\
&=& x(x+a) (x+a^{p^k})  \cdots (x+a^{p^{(p^k-2)k}}) \\
&=& x \left( (x+a)(x+a\zeta) \cdots (x+a\zeta^{n-1}) \right)^{(p^k-1)/n}
\end{eqnarray*}

If $a^{p^k-1} = -1$, then $h_a(x) = x(x^2-a^2)^{(p^k-1)/2}$. We note that $a^{p^k-1} =-1$ implies that $a^2 \in \F_{p^k}^*$ but $a\not\in \F_{p^k}$.  For any $\frac{p^k-1}{2}$-th root of unity $\xi$, we have $(\xi-a^2)^{p^k-1} =1$.  
By Corollary~\ref{power},   $h_a(x)$ is a PP of $\F_{p^k}$ and thus $a^{-1} x^d$ is a CPP over $\F_{p^{nk}}$.  Similarly, 
 If   $a^{p^k-1} \in \mu_{n} \setminus \{1, -1\}$, then $h_a(x) = x (x^n + a^n)^{(p^k-1)/n}$. 
We note that $a^{p^k-1} \in \mu_{n} \setminus \{1\}$ implies that $a^n \in \F_{p^k}^*$ but $a\not\in \F_{p^k}$.  Therefore $x^n+a^n =0$ has no solution in $\F_{p^k}$.  Hence $(\xi +a^n)^{p^k-1} - 1$ for all $\xi \in \mu_{\ell}$.     By Corollary~\ref{power},  $h_a(x)$ is a PP of $\F_{p^k}$ and thus the result holds.
\end{proof}

The following description of CPP monomials is obvious for more general $n \mid p^k-1$. 

\begin{thm}\label{thm5}
Let $p$ be an odd prime,  $n \geq 3, k$ be positive integers such that $n \mid p^{k}-1$.   Let
$d = \frac{p^{nk}-1}{p^k-1} + 1$. Let  $a^{p^k-1}$ be a primitive $n$-th root of unity in $\F_{p^{nk}}$. Then  $a^{-1} x^d$ is a CPP over $\F_{p^{nk}}$ if and only if  
\[
  [(a^n + c^i)(a^n+c^j)^{-1}]^{r} \neq  c^{j-i} ~ for~ all~ 0\leq i < j < (p^k-1)/n
\]
where $c$ is a fixed primitive $(\frac{p^k-1}{n})$-th root of unity in $\F_{p^{k}}$. 
\end{thm} 
\begin{proof}
Similarly as above,  $h_a(x) = x(x^n + a^n)$.  It is well known that $h_a(x)=  x(x^n + a^n)$ is a PP of 
$\F_{p^k}$ if and only if   $(-a^n)^{(p^k-1)/n}  \neq 1$ and  
\[
  [(a^n + c^i)(a^n+c^j)^{-1}]^{n} \neq  c^{j-i} ~ for~ all~ 0\leq i < j < (p^k-1)/n
\]
where $c$ is a fixed primitive $(\frac{p^k-1}{n})$-th root of unity in $\F_{p^{k}}$ (see for example, Excercise 7.11 in \cite{LN:97}) . Because  $a^{p^k-1}$ is a primitive $n$-th root of unity in $\F_{p^{nk}}$ and $n\geq 3$,  $(-a)^{(p^k-1)/n} \neq 1$ always holds. 

\end{proof}

Finally we obtain the following  class of CPP monomials. 
\begin{cor}
Let $p$ be an odd prime, $n \geq 3$, $k$ be a positive integer such that $n \ell =p^k-1$.   Let
$d = \frac{p^{nk}-1}{p^k-1} + 1$. Let  $a^{p^k-1}$ be a primitive $n$-th root of unity in $\F_{p^{nk}}$.
 If there exist an integer $\lambda$ such that $(z +a^n/z)^n =z^{\lambda}$ for every $2\ell $-th root of unity $z$, then  $a^{-1} x^d$ is a CPP over $\F_{p^{nk}}$ if and only if $(2+n+\lambda, 2\ell) \leq 2$. 

\end{cor}
\begin{proof}
As in Theorem~\ref{thm5}, we need to prove $h_a(x) = x(x^n + a^n)$ is a PP of $\F_{p^k}$.  We use Corollary~\ref{lemma} again,  $h_a(x) = x(x^n + a^n)$ is a PP of  $\F_{p^k}$  if and only if 
$g(x) = x(x+a^n)^{n}$  permute $\mu_{\ell}$, the set of all $\ell$-th roots of unity.   Obviously, $(-a^n)^{(p^k-1)/n} \neq 1$.  Then  $g(x) = x(x+a^n)^{n}$  permute $\mu_{\ell}$ if and only if $g(x^2)=x^2(x^2+a^n)^{n}$  is injective on $\mu_{2 \ell}/\mu_{2}$.  For any $z\in \mu_{2\ell}$, we must have 
$g(z^2)=z^2(z^2+a^n)^{n} = z^{2+n} (z+ a^n/z)^n = z^{2+n+\lambda}$. Therefore $g(x^2)=x^2(x^2+a^n)^{n}$  is injective on $\mu_{2 \ell}/\mu_{2}$ if and only if $(2+n+\lambda, 2\ell)\leq 2$. 
\end{proof}

\end{document}